\documentclass[12pt]{article}

\usepackage{geometry} 
\usepackage{amsthm,amssymb}
\usepackage{amsmath,amsfonts,latexsym}
\usepackage{latexsym}
\usepackage{float}
\usepackage{xcolor}
\usepackage{url}

\def\titlerunning#1{\gdef\titrun{#1}}
\makeatletter
\def\author#1{\gdef\autrun{\def\and{\unskip, }#1}\gdef\@author{#1}}
\def\address#1{{\def\and{\\\hspace*{18pt}}\renewcommand{\thefootnote}{}%
\footnote {#1}}% 
\markboth{\autrun}{\titrun}}
\makeatother
\def\email#1{\hspace*{4pt}{\em e-mail}: #1}

\newtheorem{thm}{Theorem}[section]
\newtheorem{prop}[thm]{Proposition}
\newtheorem{lemma}[thm]{Lemma}
\newtheorem{cor}[thm]{Corollary}
\theoremstyle{definition}
\newtheorem{rem}[thm]{Remark}
\newtheorem{defi}[thm]{Definition}
\newtheorem{example}[thm]{Example}

\begin{document}

\titlerunning{}
\title{New constructions of Deza digraphs}

\author{Dean Crnkovi\' c, Hadi Kharaghani, Sho Suda and Andrea \v Svob}

\maketitle

\address{D. Crnkovi\'{c}, A. \v Svob: Department of Mathematics, University of Rijeka, Radmile Matej\v{c}i\'c 2, 51000 Rijeka, Croatia;
\email{\{deanc,asvob\}@math.uniri.hr}
\and
H. Kharaghani: Department of Mathematics and Computer Science, University of Lethbridge, Lethbridge, Alberta, T1K 3M4, Canada;
\email{kharaghani@uleth.ca} 
\and
S. Suda: Department of Mathematics, National Defense Academy of Japan, 2-10-20 Hashirimizu, Yokosuka, Kanagawa, 239-8686, Japan;
\email{ssuda@nda.ac.jp} 
%\and 
%Corresponding author: D. Crnkovi\'c
}

\begin{abstract}
Deza digraphs were introduced in 2003 by Zhang and Wang as directed graph version of Deza graphs, that also generalize the notion of directed strongly regular graphs.
In this paper we give several new constructions of Deza digraphs. Further, we introduce twin and Siamese twin (directed) Deza graphs and construct several examples.
Moreover, we classify directed Deza graphs with parameters $(n,k,b,a,t)$ having the property that $b=t$. 
Finally, we introduce a variation of directed Deza graphs and provide a construction from finite fields.  
\end{abstract}

\bigskip



\section{Introduction}\label{intro}

Deza graphs were introduced in \cite{deza}, as a generalization of strongly regular graphs. For recent results on Deza graphs we refer the readers to \cite{Deza-Haemers,Deza-Goryainov,Deza-Kabanov}.
Further, Duval in \cite{duval} introduced directed strongly regular graphs as directed graph version of strongly regular graphs. 
Deza digraphs, a directed graph version of Deza graphs, were introduced in \cite{dideza} and further studied in \cite{wangI, wangII}.
In this paper we give several new constructions of Deza digraphs and classify directed Deza graphs with parameters $(n,k,b,a,t)$ having the property that $b=t$. 
Further, we introduce twin and Siamese twin (directed) Deza graphs and construct several examples.
Finally, we introduce a variation of directed Deza graphs, called directed Deza graphs of type II, and provide a construction from finite fields.
Directed Deza graphs of type II include divisible design digraphs (see \cite{ddd}). Note that a divisible design digraph is not necessarily a directed Deza graph.

Let $D=(V,A)$ be a directed graph (digraph), where $V$ is the set of vertices and $A$ is the set of arcs (directed edges). 
For $u, v \in V$ we will write $u \rightarrow v$ if there is an arc from $u$ to $v$,
and we will say that $u$ dominates $v$ or that $v$ is dominated by $u$. 
We will also write $u \sim v$ if $u \rightarrow v$ and $v \rightarrow u$. In this case, we will count these as one undirected edge, and say that $u$ is adjacent to $v$.
A digraph $D$ is called regular of degree $k$ if each vertex of $D$ dominates exactly $k$ vertices and is dominated by exactly $k$ vertices. 
The digraphs that we use will not have more than one arc from one vertex to another, and will not have any arcs from a vertex to itself.

A digraph $D$ on $n$ vertices is characterized by the $n \times n$ $(0,1)$-matrix $M= [m_{i,j}]$, where $m_{ij}=1$ if and only if $i \rightarrow j$ (or $i \sim j$), called the adjacency matrix of $D$.
If the adjacency matrix $M$ of a digraph $D$ has the property that $M+M^t$ is a $(0,1)$-matrix, the $D$ is called asymmetric.
If $M$ satisfies the conditions
$$MJ_n=J_nM=kJ_n,$$
$$M^2=tI_n+ \lambda M + \mu (J_n-I_n-M), $$
where $I_n$ is the identity matrix of order $n$, and $J_n$ is the $n \times n$ matrix of all 1's, then $D$ is called s directed strongly regular graph (DSRG) with parameters $(n,k,\lambda,\mu, t)$.

\bigskip

The rest of the paper is organized as follows. In Section \ref{Deza-defi} we give basic properties of Deza digraphs, in Section \ref{Deza-con} we give
constructions of directed Deza graphs and classify directed Deza graphs with parameters $(n,k,b,a,t)$ having the property that $b=t$, 
where $t$ is the number of vertices adjacent to a vertex of the directed Deza graph, in Section \ref{twin} we introduce twin and Siamese twin (directed) Deza (reflexive) graphs 
and construct several examples, and in Section \ref{typeII} we consider a variation of directed Deza graphs, called directed Deza graphs of type II, and show some of its construction. 

\section{Preliminaries} \label{Deza-defi}

Directed Deza graphs were introduced in \cite{dideza} and further studied in \cite{wangI, wangII}.

\begin{defi}
Let $n$, $k$, $b$, $a$, and $t$ be integers such that $0 \leq a \leq b \leq k \leq n$ and $0 \leq t \leq k $. A digraph $D =( V,A)$ is a directed $(n, k, b, a, t)$-Deza graph if
\begin{enumerate}
	\item $|V| = n$. 
	\item Every vertex has in-degree and out-degree $k$, and is adjacent to $t$ vertices.
	\item Let $u$ and $v$ be distinct vertices. The number of vertices $w$ such that $u \rightarrow w \rightarrow v$ is $a$ or $b$.
\end{enumerate}
\end{defi}

\begin{example} \label{ex1}
Let us define the matrices $M_1$ and $M_2$ as follows
\begin{displaymath}  M_1=   \left[
\begin{tabular}{rrrrrrrr}
0  & 0  &  1 &  0 &  1 & 0 & 1 & 0\\
0  & 0  &  0 &  1 &  0 & 1 & 0 & 1\\
0  & 1  &  0 &  0 &  1 & 0 & 0 & 1\\
1  & 0  &  0 &  0 &  0 & 1 & 1 & 0\\
0  & 1  &  0 &  1 &  0 & 0 & 1 & 0\\
1  & 0  &  1 &  0 &  0 & 0 & 0 & 1\\
0  & 1  &  1 &  0 &  0 & 1 & 0 & 0\\
1  & 0  &  0 &  1 &  1 & 0 & 0 & 0\\
\end{tabular}                 \right], \quad
M_2=   \left[
\begin{tabular}{rrrrrrrr}
0  & 1  &  0 &  1 &  0 & 1 & 0 & 1\\
1  & 0  &  1 &  0 &  1 & 0 & 1 & 0\\
1  & 0  &  0 &  1 &  0 & 1 & 1 & 0\\
0  & 1  &  1 &  0 &  1 & 0 & 0 & 1\\
1  & 0  &  1 &  0 &  0 & 1 & 0 & 1\\
0  & 1  &  0 &  1 &  1 & 0 & 1 & 0\\
1  & 0  &  0 &  1 &  1 & 0 & 0 & 1\\
0  & 1  &  1 &  0 &  0 & 1 & 1 & 0\\
\end{tabular}                 \right].
\end{displaymath}
The matrix $M_1$ is the adjacency matrix of a directed ($8, 3, 3, 1, 0$)-Deza graph $D_1$. Since $t=0$, the digraph $D_1$ is asymmetric. 
The matrix $M_2$ is the adjacency matrix of a directed ($8, 4, 3, 1, 1$)-Deza graph $D_2$.
\end{example}

\bigskip

Let $M$ be the adjacency matrix of a directed graph $D$ on $n$ vertices.
Then $D$ is a directed $(n, k, b, a, t)$-Deza graph if and only if $MJ_n=J_nM=kJ_n$ and $M^2 = aX +bY +tI_n$ for some (0,1)-matrices $X$ and $Y$ such that $X + Y + I_n = J_n$. 
Note that $G$ is a directed strongly regular graph if and only if $X$ or $Y$ is $M$.

\bigskip

Suppose that we have a directed Deza graph with $M$, $X$, and $Y$ as above. Then $X$ and $Y$ are adjacency matrices of digraphs. We will denote these graphs by $D_X$ and $D_Y$, and call them the Deza children of $D$. Deza children are sometimes directed Deza graphs, and sometimes even directed strongly regular graphs. 
Deza children can be used to distinguish between nonisomorphic directed Deza graphs with the same parameters.

\bigskip

In case $X$ and $Y$ are symmetric, we have a Deza graph, and that case is excluded here. This means that we require that $t < k$.

\bigskip

Let $D =( V,A)$ be a directed ($n, k, b, a, t$)-Deza graph. For vertices $u$ and $v$, let $N_{uv}$ be the number of vertices $w$ such that $u \rightarrow w \rightarrow v$. 
Further, for a vertex $u$ we define
$$\alpha=|\{v\in V : N_{uv}=a\}|, \qquad \beta=|\{v\in V : N_{uv}=b\}|.$$ 
The following statement can be found in \cite[Proposition 1.1.]{dideza}.

\begin{prop} \label{cond}
Let $D$ be a directed $(n, k, b, a, t)$-Deza graph. The numbers $\alpha$ and $\beta$ do not depend on the vertex $u$ and 

$$
\alpha=
\begin{cases}
\displaystyle\frac{b(n-1)-k^2+t}{b-a}, & a\neq b,\\
\displaystyle\frac{k^2-t}{a}, & a = b.\\
\end{cases}
$$

$$
\beta=
\begin{cases}
\displaystyle\frac{a(n-1)-k^2+t}{a-b}, & a\neq b,\\
\displaystyle\frac{k^2-t}{a}, & a = b.\\
\end{cases}
$$
\end{prop}

\begin{rem}
For $k=t$ we have a Deza graph for which conditions given in \cite[Proposition 1.1]{deza} follow from the conditions given in Proposition \ref{cond}.
\end{rem}

If $a=b$, then a directed ($n, k, b, a, t$)-Deza graph is a directed strongly regular graph. Hence, we are interested in the case $a \neq b$.
The following corollary is a direct consequence of Proposition \ref{cond}.

\begin{cor} \label{cond-cor}
Let $D$ be a directed $(n, k, b, a, t)$-Deza graph. If $a \neq b$, then $b-a$ divides $b(n-1)-k^2+t$ and $a(n-1)-k^2+t$.
If $a<b$ and $\alpha, \beta \neq 0$, then 
$$a(n-1)<k^2-t<b(n-1).$$
\end{cor}

\section{Constructions of directed Deza graphs} \label{Deza-con}

In this section we give some constructions of directed Deza graphs. Also, we classify directed ($n, k, b, a, t$)-Deza graphs with $b=t$.

\subsection{A construction from the lexicographical product}

For two digraphs $D=(V_1,A_1)$ and $H=(V_2,A_2)$ the lexicographical product digraph (or the composition) $D[H]$ is the digraph with vertex set $V_1\times V_2$ and arc set defined as follows. 
There is an arc from a vertex $(x_1, y_1)$ to a vertex $(x_2, y_2)$ in $D[H]$ if and only if either $x_1 \rightarrow x_2$ in $D$ or $x_1=x_2$ and $y_1 \rightarrow y_2$ in $H$.

Let $M_1$ be the adjacency matrix of $D$ and $M_2$ be the adjacency matrix of $H$. Then $M_1 \otimes J_{|V_2|} + I_{|V_1|} \otimes M_2$ is the adjacency matrix of $D[H]$, where $\otimes$ denotes
the Kronecker product of matrices and the vertices of $D[H]$ are ordered lexicographically. The following theorem is given in \cite[Theorem 2.2.]{dideza}.

\begin{thm} \label{lex-prod}
Let $D_1=(V_1,A_1)$ be a DSRG with parameters $(n,k,\lambda,\mu, t)$ and $D_2=(V_2,A_2)$ be a directed $(n', k', b, a, t')$-Deza graph. Then $D_1[D_2]$ is a $(k'+kn')$-regular digraph on $nn'$ vertices.
It is a directed Deza graph if and only if
$$| \{ a+kn', b+kn', \mu n', \lambda n' +2k' \} | \le 2.$$
\end{thm}

\begin{example} \label{ex-b=t}
Let $D_1=(V_1,A_1)$ be a DSRG with parameters $(n,k,\lambda,\lambda,t)$, and $D_2=(V_2,A_2)$ be the empty digraph on $n'$ vertices (i.e., $A_2=\emptyset$). 
Then $D_1[D_2]$ is a directed Deza graph with parameters $(nn',kn',tn',\lambda n',tn')$.
\end{example}

\bigskip

In the next theorem we show that the digraphs from Example \ref{ex-b=t} are characterized by their parameters. The proof of Theoem \ref{thm-b=t} follows the proof of \cite[Theorem 2.6.]{deza}.

\begin{thm} \label{thm-b=t}
Let $D$ be a directed $(n, k, b, a, t)$-Deza graph. Then $b=t$ if and only if $D$ is isomorphic to $D_1[D_2]$, 
where $D_1$ is a DSRG with parameters $(n_1, k_1, a_1, a_1, t_1)$, and $D_2$ is an empty digraph of $n_2$ vertices. Moreover, the parameters of $D$ satisfy
$$n=n_1n_2, \qquad k=k_1n_2, \qquad b=t=t_1n_2, \qquad a=\lambda n_2.$$
\end{thm}

\begin{proof}
If $D_1=(V_1,A_1)$ is a DSRG with parameters $(n_1, k_1, a_1, a_1, t_1)$, $D_2=(V_2,A_2)$ is an empty digraph ($A_2=\emptyset$) of $n_2$ vertices and $D=D_1[D_2]$, 
then the statement follows from Theorem \ref{lex-prod} and Example \ref{ex-b=t}.

Suppose that $D$ is a directed $(n, k, b, a, t)$-Deza graph and $b=t$. Let us consider the equivalence relation $R$ on the set of vertices of $D$ given by
$$uRv \ {\rm iff} \ N_{uv}=b.$$
Proposition \ref{cond} implies that each equivalence class has the same size, namely $\beta +1$. Let $D_2$ be the empty digraph on $n_2=\beta +1$ vertices. 
Let us define the digraph $D_1$. The vertices of $D_1$ are the equivalence classes of the relation $R$, where $C_1 \rightarrow C_2$ in $D_1$ if and only if there is an arc in $D$
which joins a vertex from $C_1$ to a vertex from $C_2$.

Let us show that $D$ is isomorphic to $D_1[D_2]$. Let $V_2= \{ u_1, \ldots u_{n_2} \}$, and let the equivalence class $C_i$ be $\{ v_{i1}, v_{i2}, \ldots , v_{i{n_2}} \}$.
The mapping $f$ given by $f(C_i,u_j)=v_{ij}$ is a digraph isomorphism from $D_1[D_2]$ to $D$.

It remains to show that $D_1$ is a DSRG with parameters $(n_1, k_1, a_1, a_1, t_1)$. Since $D_2$ is the empty digraph on $n_2$ vertices, the adjacency matrix of $D_1[D_2]$ is $M_1 \otimes J_{n_2}$,
where $M_1$ is the adjacency matrix of $D_1$. Therefore, since $\beta +1$, it follows that $D_1$ is a DSRG with parameters $(n_1, k_1, a_1, a_1, t_1)$. Moreover,
$D$ is a Deza digraph with parameters $(n_1n_2, k_1n_2, t_1n_2, \lambda n_2, t_1n_2)$.
\end{proof}

\subsection{A construction from association schemes}

We assume that the reader is familiar with the basic facts of theory of association schemes. 
For background reading in theory of association schemes we refer the reader to \cite{bann-ito}. 

Let $X$ be a finite set of size $n$ and $\mathcal{R}= \{ R_0,R_1,\ldots,R_d\}$ be relations defined on the set $X$. Let $\mathcal{A}=\{A_0, A_1, \ldots, A_d \}$  be the set of $(0,1)$-adjacency matrices such that $[A_i]_{xy}=1$ if and only if $(x,y) \in R_i$. Then a pair $(X,\mathcal{R})$ is called an {\it association scheme} with $d$ classes if

\begin{enumerate}
 \item $A_0=I$,
 \item $\sum_{i=0}^d A_i=J$,
 \item $A_i^t \in \mathcal{A}$, for all $i \in \{ 0,1,\ldots,d \} $,
 \item For any $i,j,k\in\{0,1,\ldots,d\}$, there exists a non-negative integer $p_{i,j}^k$ such that $A_i A_j=\sum_{k=0}^d p_{i,j}^k A_k$. 
 \item For any $i,j\in\{1,\ldots,d\}$, $A_i A_j=A_j A_i $. 
\end{enumerate}

We view $A_1,\ldots, A_d$ as adjacency matrices of directed graphs $D_1, \ldots, D_d$, with common vertex set. An association scheme is symmetric if each matrix in it is symmetric, and non-symmetric otherwise. A method of constructing Deza graphs from symmetric association schemes is given in \cite[Theorem 4.2.]{deza}. Deza digraphs can be constructed from association scheme in a similar way,
as shown in the following theorem.

\begin{thm} \label{thm-as}
Let $(X,\mathcal{R})$ be an association scheme, and $F\subset \{1,2,\ldots, d\}$. Let $D$ be the directed graph with adjacency matrix $\sum_{f\in F}A_f$. 
Then $D$ is a directed Deza graph if and only if $$\displaystyle\sum_{f,g\in F}p_{f,g}^k$$ takes on at most two values, when $k\in \{1,\ldots, d\}$. 
\end{thm}

\begin{proof}
Each $A_i$ can be regarded as an adjacency matrix of a regular digraph (see \cite{hanaki}). Since $\sum_{i=0}^d A_i=J$, any sum is an adjacency matrix of a regular digraph, 
i.e. every vertex has constant in-degree and out-degree. Moreover, every vertex of $D$ is adjacent with 
$$t=\displaystyle\sum_{f\in F}p_{f,f}^0$$ 
vertices. Let $u,v\in V(D)$, and $d(u,v)=k$. Then $$N_{uv}=\displaystyle\sum_{f,g\in F}p_{f,g}^k.$$ When these numbers take on at most two values $D$ is a directed Deza graph, which completes the proof.
\end{proof}
\begin{example}
The adjacency matrix $A$ of a doubly regular tournament with $4t+3$ vertices satisfies that $A^2=(t+1)(J-I)-A$, which implies that a doubly regular tournament is a directed 
$(4t+3,2t+1,t+1,t,0)$-Deza graph. Note that a doubly regular tournament is equivalent to a non-symmetric association scheme with $2$ classes. 
\end{example}
\begin{example}
For a non-symmetric association scheme with $3$ classes such that $A_1^t=A_2$, it holds that $p_{1,1}^1=p_{1,1}^2$. Therefore, the digraph $D_1$ with adjacency matrix $A_1$ is a directed 
$(|X|,p_{1,2}^0,b,a,p_{1,1}^0)$-Deza graph where $\{a,b\}=\{p_{1,1}^1,p_{1,1}^3\}$.   For construction, see \cite{GC, KST}. 
\end{example}
\begin{example}
For a non-symmetric association scheme with $5$ classes constructed in \cite[Theorem~6.3]{KS} with $A_1^t=A_2$, it holds that $\{p_{1,1}^k : 1\leq k\leq 5\}=\{(2n-1)^2 (n-1) n,(2 n-1)^2 n^2\}$. 
Therefore, $A_1$ is a directed Deza graph. 
\end{example}

\subsection{A construction from Hadamard matrices}

We say that a $( 0,1 )$-matrix $X$ is skew if $X+X^t$ is a $( 0,1 )$-matrix.
A Hadamard matrix $H$ of order $n$ is called skew-type if $H+H^t=2I_n$.

Let $D$ be a regular asymmetric digraph of degree $k$ on $v$ vertices.
$D$ is called a divisible design digraph (DDD for short) 
with parameters $(v,k, \lambda_1, \lambda_2, m,n)$ if the vertex set 
can be partitioned into $m$ classes of size $n$, such that
for any two distinct vertices $x$ and $y$ from the same class, 
the number of vertices $z$ that dominates or being dominated by both 
$x$ and $y$ is equal to $\lambda_1$,
and for any two distinct vertices $x$ and $y$ from different classes, 
the number of vertices $z$ that dominates or being dominated by both 
$x$ and $y$ is equal to $\lambda_2$.

An incidence structure with $v$ points and the constant block size $k$ is a 
(group) divisible design with parameters $(v,k, \lambda_1, \lambda_2, m,n)$ 
whenever the point set can be partitioned into $m$ classes of size $n$, 
such that two points from the same class are incident with exactly $\lambda_1$ 
common blocks, and two points from different classes are incident with exactly 
$\lambda_2$ common blocks. A divisible design $D$ is said to be 
symmetric (or to have the dual property) if the dual of $D$ is a 
divisible design with the same parameters as $D$.
If $D$ is a divisible design digraph with parameters $(v,k, \lambda_1, \lambda_2, m,n)$ then its 
adjacency matrix is the incidence matrix of a symmetric divisible design $(v,k, \lambda_1, \lambda_2, m,n)$.

The following construction of directed Deza graphs follows from the construction of divisible design digraphs given in \cite[Theorem 10.]{ddd}.
By $O_n$ we denote the $n \times n$ zero-matrix.

\begin{thm} \label{had1}
Let $H$ be a skew Hadamard matrix of order $4u$. % with diagonal entries equal to 1. 
Then there exists a directed Deza graph with parameters $(8u,4u-1,4u-1,2u-1,0)$.
\end{thm}

\begin{proof}
Replace each diagonal entry of $H$ by $O_2$, each entry value 1 of $H$ by $I_2$, and each entry value $-1$ by $J_2-I_2$. By \cite[Theorem 10.]{ddd}, the obtained matrix is the adjacency matrix of a
DDD with parameters $(8u,4u-1,0,2u-1,4u,2)$. By the definition of a DDD, the obtained digraph is asymmetric. By the construction, each row (and column) of the Hadamard matrix correspond to a class of
size two of vertices of the DDD. Let $r_i$ be the row of the adjacency matrix of the DDD corresponding to the vertex $v_i$, and $c_i$ be the column of the adjacency matrix corresponding to $v_i$. 
Since the digraph is asymmetric, the dot product $r_i \cdot c_i=0$. If the vertex $v_j$ belongs to the same class as $v_i$, then $r_i \cdot c_j=4u-1$, i.e. $r_i=c_j$.
Since the adjacency matrix $M$ of $D$ is the incidence matrix of a symmetric divisible design and there is a column of $M$ equal to $r_i$, the statement of the theorem follows.
\end{proof}

\begin{example} \label{ex-had1}
The matrix $M_1$ from Example \ref{ex1} can be obtained by Theorem \ref{had1} using the Hadamard matrix
\begin{displaymath}
H=   \left[
\begin{tabular}{r r r r}
 1 &  1  &  1 &  1 \\
- &  1  &  1 & - \\
- & -  &  1 &  1 \\
- &  1  & - &  1
\end{tabular}                 
\right].
\end{displaymath}
\end{example}

\section{Twin and Siamese twin (directed) Deza graphs} \label{twin}

Twin designs and Siamese twin designs were studied in \cite{Hadi-twin, Hadi-Siamese}. The concepts of twin and Siamese twin structures are developed for strongly regular graphs too
(see \cite{Janko-Hadi}). In this paper we introduce twin and Siamese twin (directed) Deza graphs and construct several examples.

A $(0,\pm 1)$-matrix $T$ is called a twin (directed) Deza graph, if $T=K-L$, where $K,L$ are non-zero disjoint $(0,1)$-matrices and both $K$ and $L$ are
adjacency matrices of (directed) Deza graphs with the same parameters. 

A $(0,\pm 1)$-matrix $S$ is called a Siamese twin (directed) Deza graph sharing the entries of $N$,
if $S=N+K-L$, where $N,K,L$ are non-zero disjoint $(0,1)$-matrices and both $N+K$ and $N+L$ are
adjacency matrices of (directed) Deza graphs with the same parameters. 

Let $G$ be a graph on $v$ vertices. A reflexive graph $RG$ is obtained from $G$ by including a loop at every vertex. If $A$ is the adjacency matrix of the graph $G$, then $A+I_v$ is the
adjacency matrix of $RG$. In the following theorem we give constructions of twin Deza graphs and Siamese twin Deza reflexive graphs. 

\begin{thm} \label{twin-graphs}
Let there exist a Hadamard matrix of order $n$. Then there exist a twin Deza graph with parameters $((2n-1)n,(n-1)n,\frac{n(n-1)}{2},\frac{n(n-2)}{2})$, 
and a Siamese twin Deza reflexive graph with parameters $((2n-1)n,n^2,\frac{n(n+1)}{2},\frac{n^2}{2})$.
\end{thm}

\begin{proof}
Let $H$ be a normalized Hadamard matrix of order $n$. Let $r_i$ be the $i$-th row of the matrix $H$ and let $C_i=r_i^t r_i$, $i=1,2,3, \ldots ,n$.
Further, let $C$ be the circulant matrix of order $2n-1$ with the first row $(1, 2, 3, \ldots ,n,n,n-1, \ldots , 2)$. 
Note that every pair of rows of this matrix have exactly one common entry in the same column. Now replace $i$ with $C_i$ in $C$. We get a $(2n-1)n \times (2n-1)n$ matrix $K$, which is a $(1,-1)$-matrix 
such that $K^2$ has all its off-diagonal entries $n$ or $-n$. By changing the diagonal blocks of size $n \times n$ in $K$ to the zero matrices, and split the remaining matrix into $A-B$, 
where $A$ and $B$ are disjoint $(0,1)$-matrices, we obtain the adjacency matrices $A$ and $B$ of two graphs. The rows and the columns of $A$ an $B$ are divided into $2n-1$ groups of size $n$, according to
the circulant matrix of order $2n-1$. From \cite[Lemma 3.]{Hadi-Siamese} it follows that for two rows $r_i$ and $r_j$, $i \neq j$, of $A$ (or $B$) belonging to the same group 
$r_i \cdot r_j =\frac{n(n-2)}{2}$. The properties of the matrices $C_i$, $i=1,2,3, \ldots ,n$, given in the proof of \cite[Theorem 1.]{Hadi-Siamese} lead us to conclusion that for two rows 
$r_i$ and $r_j$ from different groups the dot product $r_i \cdot r_j$ is equal to $\frac{n(n-2)}{2}$ or $\frac{n(n-1)}{2}$. Hence, $A$ and $B$ are adjacency matrices of Deza graphs with parameters 
$((2n-1)n,(n-1)n,\frac{n(n-1)}{2},\frac{n(n-2)}{2})$ and $K$ is the adjacency matrix of a twin Deza graph.

If we now add the matrix $I_{2n-1} \otimes C_1$ to each of $A$ or $B$, we get two Deza reflexive graphs with parameters $((2n-1)n,n^2,\frac{n(n+1)}{2},\frac{n^2}{2})$ sharing the cliques of size $n$.
\end{proof}

\begin{example} \label{ex-twin}
Let us take the matrix
\begin{displaymath}
H=   \left[
\begin{tabular}{r r r r}
 1 &  1  &  1 &  1 \\
 1 &  1  & - & - \\
 1 & -  &  1 & - \\
 1 & -  & - &  1
\end{tabular}                 \right],
\end{displaymath}
which is a normalized Hadamard matrix of order 4. Then $r_1=(1,1,1,1)$, $r_2=(1,1,-,-)$, $r_3=(1,-,1,-)$, $r_4=(1,-,-,1)$, and
\begin{displaymath}
C_1=   \left[
\begin{tabular}{r r r r}
 1 &  1  &  1 &  1 \\
 1 &  1  &  1 &  1 \\
 1 &  1  &  1 &  1 \\
 1 &  1  &  1 &  1
\end{tabular}                 \right], \qquad
C_2=   \left[
\begin{tabular}{r r r r}
 1 &  1  & - & - \\
 1 &  1  & - & - \\
- & -  &  1 &  1 \\
- & -  &  1 &  1
\end{tabular}                 \right],
\end{displaymath}

\begin{displaymath}
C_3=   \left[
\begin{tabular}{r r r r}
 1 & -  &  1 & - \\
- &  1  & - &  1 \\
 1 & -  &  1 & - \\
- &  1  & - &  1 
\end{tabular}                 \right], \qquad
C_4=   \left[
\begin{tabular}{r r r r}
 1 & -  & - &  1 \\
- &  1  &  1 & - \\
- &  1  &  1 & - \\
 1 & -  & - &  1  
\end{tabular}                 \right].
\end{displaymath}

Let $C$ be the circulant matrix of order 7 with the first row $(1,2,3,4,4,3,2)$. By applying Theorem \ref{twin-graphs} one gets a twin Deza graph with parameters $(28,12,6,4)$, and a
Siamese twin Deza reflexive graph with parameters $(28,16,10,8)$. 
\end{example}

\begin{thm} \label{twin-digraphs}
Let there exist a Hadamard matrix of order $n$. Then there exist a Siamese twin directed Deza reflexive graph with parameters $((2n-1)n,n^2,\frac{n(n+1)}{2},\frac{n^2}{2},n)$.
\end{thm}

\begin{proof}
Let $H$ be a normalized Hadamard matrix of order $n$ and let the matrices $C_i$, $i=1,2, \ldots , n$ be defined as in the proof of Theorem \ref{twin-graphs}.
Further, let $D$ be the circulant matrix with the first row $(1, 2, 3, \ldots ,n,-n,-n+1,\ldots, -2)$. Now replace $i$ with $C_i$ in $C$ if $i > 0$, and replace $i$ with $-C_{|i|}$
if $i<0$. We get a $(2n-1)n \times (2n-1)n$ $(1,-1)$-matrix $K'$ such that $K'K'^t$ has all its off-diagonal entries $n$ or $-n$. By changing the block diagonals of $K'$ to the zero matrices, 
we get a matrix which splits as $A-B$, where $A$ and $B$ are $(0,1)$-matrices. If we now add the matrix $I_{2n-1} \otimes C_1$ to each of $A$ or $B$, we get adjacency matrices of two directed  Deza reflexive graphs with parameters $((2n-1)n,n^2,\frac{n(n+1)}{2},\frac{n^2}{2},n)$.
\end{proof}

\begin{rem}
The matrices $A$ and $B$ from the proof of Theorem \ref{twin-digraphs} are the adjacency matrices of divisible design digraphs with parameters $((2n-1)n,(n-1)n,\frac{n(n-2)}{2},\frac{n(n-1)}{2},2n-1,n)$.
\end{rem}

\section{A variation of directed Deza graphs} \label{typeII}

Finally, we introduce  a slightly different variation of directed Deza graphs. 

\begin{defi}
Let $n$, $k$, $b$, and $a$ be integers such that $0 \leq a \leq b \leq k \leq n$. A digraph $D =( V,A)$ is a directed $(n, k, b, a)$-Deza graph of type II if
\begin{enumerate}
	\item $|V| = n$. 
	\item Every vertex has in-degree and out-degree $k$.
	\item Let $u$ and $v$ be distinct vertices. The number of vertices $w$ such that $u \rightarrow w \leftarrow v$ and the number of vertices $w$ such that $u \leftarrow w \rightarrow v$ coincide and are $a$ or $b$.
\end{enumerate}
\end{defi}

Let $M$ be the adjacency matrix of a directed graph $D$ on $n$ vertices.
Then $D$ is a directed $(n, k, b, a)$-Deza graph of type II if and only if $MJ_n=J_nM=k J_n$ and $M M^t = M^t M = aX +bY +k I_n$ for some (0,1)-matrices $X$ and $Y$ such that $X + Y + I_n = J_n$. 

Note that a doubly regular asymmetric digraph (DRAD) with parameters $(v,k,\lambda)$ (see \cite{Ito-DRADs}) is a directed $(v, k, \lambda, \lambda)$-Deza graph of type II.
Divisible design digraphs (see \cite{ddd}) are natural generalization of doubly regular asymmetric digraphs.
A divisible design digraph with parameters $(v,k,\lambda_1,\lambda_2,m,n)$ is a directed $(v, k, \lambda_1, \lambda_2)$-Deza graph of type II.

Let $D =( V,A)$ be a directed ($n, k, b, a$)-Deza graph of type II. For vertices $u$ and $v$, let $N_{uv}$ be the number of vertices $w$ such that $u \rightarrow w \leftarrow v$. 
Further, for a vertex $u$ we define
$$\alpha=|\{v\in V : N_{uv}=a\}|, \qquad \beta=|\{v\in V : N_{uv}=b\}|.$$ 
%The following statement can be found in \cite[Proposition 1.1.]{dideza}.

\begin{prop} \label{cond-II}
Let $D$ be a directed $(n, k, b, a)$-Deza graph of type II. The numbers $\alpha$ and $\beta$ do not depend on the vertex $u$ and 

$$
\alpha=
\begin{cases}
\displaystyle\frac{b(n-1)-k^2+k}{b-a}, & a\neq b,\\
\displaystyle\frac{k^2-k}{a}, & a = b.\\
\end{cases}
$$

$$
\beta=
\begin{cases}
\displaystyle\frac{a(n-1)-k^2+k}{a-b}, & a\neq b,\\
\displaystyle\frac{k^2-k}{a}, & a = b.\\
\end{cases}
$$
\end{prop}
\begin{proof}
Let ${\bf 1}$ be the all-one column vector. We show that $X{\bf 1}=\alpha {\bf 1}$ and $Y{\bf 1}=\beta  {\bf 1}$ for the desired $\alpha,\beta$.  

Multiply $MM^t=aX+bY+kI$ by ${\bf 1}$ and use $M{\bf 1}=M^t{\bf 1}=k{\bf 1}$, we have $aX{\bf 1}+bY{\bf 1}=(k^2-k){\bf 1}$. 
On the other hand,  it follows from $X+Y+I=J$ that $X{\bf 1}+Y{\bf 1}=(n-1){\bf 1}$.  
Therefore, we obtain that $(b-a)X{\bf 1}=(b(n-1)-k^2+k){\bf 1}$ and $(a-b)Y{\bf 1}=(a(n-1)-k^2+k){\bf 1}$.  

When $a\neq b$, $X{\bf 1}=\frac{b(n-1)-k^2+k}{b-a}{\bf 1}$ and $Y{\bf 1}=\frac{a(n-1)-k^2+k}{a-b}{\bf 1}$ and thus $\alpha,\beta$ are shown to be the desired values. 

When $a=b$, assume that $X=J-I$ and $Y=O$. Then $X{\bf 1}=(n-1){\bf 1}$ and $a(n-1)-k^2+k=0$, which shows that $\alpha=\beta=\frac{k^2-k}{a}$. 
\end{proof}

The following theorem can be proved in a similar way as Theorem \ref{lex-prod} and \cite[Proposition 2.3.]{deza}.

\begin{thm} \label{lex-prod-II}
Let $G_1$ be an $(n,k, \lambda, \mu)$-SRG and $D_2=(V_2,A_2)$ be a directed $(n', k', b, a)$-Deza graph of type II. Then $G_1[D_2]$ is a $(k'+kn')$-regular digraph on $nn'$ vertices. 
It is a directed Deza graph of type II if and only if
$$| \{ a+kn', b+kn', \mu n', \lambda n' +2k' \} | \le 2.$$
\end{thm}

It is straightforward to check that the following statement holds.

\begin{thm} \label{exII-b=k}
Let $D_1=(V_1,A_1)$ be a digraph with the adjacency matrix $M$ which is the incidence matrix of a symmetric $(n,k,\lambda)$ design. 
Further, let $D_2=(V_2,A_2)$ be the empty digraph on $n'$ vertices (i.e., $A_2=\emptyset$). 
Then $D_1[D_2]$ is a directed Deza graph of type II with parameters $(nn',kn',kn',\lambda n')$.
\end{thm}

Directed Deza graph of type II with parameters $(n,k,k,a)$ are recognizable by their parameters, as is seen in the following theorem.
Theorem \ref{thmII-b=k} can be proved in a similar way as Theorem \ref{thm-b=t} and \cite[Theorem 2.6.]{deza}, and therefore we omit its proof.

\begin{thm} \label{thmII-b=k}
Let $D$ be a directed $(n, k, b, a)$-Deza graph of type II. Then $b=k$ if and only if $D$ is isomorphic to $D_1[D_2]$, 
where $D_1$ is a digraph with the adjacency matrix which is the incidence matrix of a symmetric $(n_1, k_1, a_1)$ design, and $D_2$ is an empty digraph of $n_2$ vertices. 
Moreover, the parameters of $D$ satisfy
$$n=n_1n_2, \qquad k=b=k_1n_2, \qquad a=\lambda n_2.$$
\end{thm}

\bigskip

Next we construct a directed Deza graph of type II from the finite fields of odd characteristic. 
From now on let $q=p^m$ be a power of an odd prime $p$. 
We denote by $\mathbb{F}_q$ the finite field of $q$ elements.   
Let $H_q$ be the multiplicative table of $\mathbb{F}_q$, i.e., $H_q$ is a $q\times q$ matrix with rows and columns indexed by the elements of $\mathbb{F}_q$ with $(\alpha,\beta)$-entry equal to $\alpha \cdot \beta$. 
Then the matrix $H_q$ is a generalized Hadamard matrix with parameters $(q,1)$ over the additive group of $\mathbb{F}_q$. 
Letting $G$ be an additively written finite abelian  group of order $g$, 
a square matrix $H=(h_{ij})_{i,j=1}^{g\lambda}$ of order $g\lambda$ with entries from $G$ is called a {\it generalized Hadamard matrix with the parameters $(g,\lambda)$} over $G$ 
if for all distinct $i,k\in\{1,2,\ldots,g\lambda\}$, the multiset $\{h_{ij}-h_{kj}: 1\leq j\leq g\lambda\}$ contains exactly $\lambda$ times of each element of $G$. 

Let $\phi$ be a permutation representation of the additive group of $\mathbb{F}_q$ defined as follows.  
Since $q=p^m$, we view the additive group of $\mathbb{F}_q$ as $\mathbb{F}_p^m$. 
Define $U=\text{circ}(0,1,0,\ldots,0)$, a circulant matrix with the first row $(0,1,0,\ldots,0)$, and a group homomorphism $\phi:\mathbb{F}_{q}\rightarrow GL_{q}(\mathbb{R})$ as $\phi((x_i)_{i=1}^m)= \otimes_{i=1}^m U^{x_i}$, where $\otimes$ denotes the tensor product. 

From the generalized Hadamard matrix $H_q$ and the permutation representation $\phi$, we construct $q^2$ auxiliary matrices; 
for each $\alpha,\alpha'\in \mathbb{F}_q$, define a $(0,1)$-matrix $C_{\alpha,\alpha'}$ to be a $q\times q$ block matrix, where $\phi(\alpha(-\beta+\beta')+\alpha')$ is placed in its $(\beta,\beta')$-entry as a submatrix:   
\begin{align*}
C_{\alpha,\alpha'}=(\phi(\alpha(-\beta+\beta')+\alpha'))_{\beta,\beta'\in\mathbb{F}_q}. 
\end{align*}
%a $q^2\times q^2$ $(0,1)$-matrix $C_{\alpha,\alpha'}$ to be 
%\begin{align*}
%C_{\alpha,\alpha'}=(\phi(\alpha(-\beta+\beta')+\alpha'))_{\beta,\beta'\in\mathbb{F}_q}. 
%\end{align*}
Further, let $x,y$ be indeterminates. We define $C_{x,\alpha},C_{y,\alpha}$ by $C_{x,\alpha}=O_{q^2}$ and $C_{y,\alpha}=\phi(\alpha)\otimes J_q$ for $\alpha\in\mathbb{F}_q$, where $O_{q^2}$ denotes the zero matrix of order $q^2$. 

Let $L$ be a $(2q+3)\times (2q+3)$ circulant matrix with the first row $(a_i)_{i=1}^{2q+3}$ satisfying 
$$
a_1=x,a_2=y, \{a_i : 2\leq i \leq q+1\}=\mathbb{F}_q, a_{2q+4-i}=a_{i+1} \text{ for }i\in\{2,\ldots,2q+1\}. 
$$
Let $S=\mathbb{F}_q \cup\{x,y\}$. 
Write $L$ as $L=\sum_{a\in S}a\cdot P_a$. %, where $P_a$ is a symmetric permutation matrix of order $q+2$.
Note that $P_x=I_{2q+3}$. 

From the $(0,1)$-matrices $C_{\alpha,\alpha'}$'s and the array $L$, we construct (directed) Deza graphs as follows. 
For $\alpha \in \mathbb{F}_q$, we define a $(2q+3)q^2\times (2q+3)q^2$ $(0,1)$-matrix $N_{\alpha}$ to be 
\begin{align*}
N_\alpha=(C_{L(a,a'),\alpha})_{a,a'\in S}=\sum_{a\in \mathbb{F}_q\cup\{y\}} P_a\otimes C_{a,\alpha}. 
\end{align*}
In order to calculate $N_{\alpha}N_{\alpha}^t$ and study more properties, we prepare a lemma on $C_{\alpha,\alpha'}$ and $P_a$.  
We fix a bijection $\varphi:\mathbb{F}_q\cup\{y\}\rightarrow \{1,\ldots,q+1\}$ such that $P_a=V^{\varphi(a)}+V^{-\varphi(a)}$ where $V$ is the shift matrix of order $2q+3$. 

\begin{lemma}\label{lem:1}
\begin{enumerate}
\item For $a\in\mathbb{F}_q\cup\{y\}$ and $\alpha\in\mathbb{F}_q$, $C_{a,\alpha}^t=C_{a,-\alpha}$. %\textcolor{red}{need to check}
\item For $\alpha\in\mathbb{F}_q$, $\sum_{a\in\mathbb{F}_q\cup\{y\}}C_{a,\alpha}=q I_q\otimes \phi(\alpha)+(J_q+\phi(\alpha)-I_q)\otimes J_q$. 
\item For $a\in\mathbb{F}_q\cup\{y\}$ and $\alpha,\alpha'\in\mathbb{F}_q$,
$C_{a,\alpha}C_{a,\alpha'}=q C_{a,\alpha+\alpha'}$.
\item For distinct $a,a'\in\mathbb{F}_q\cup\{y\}$ and $\alpha,\alpha'\in\mathbb{F}_q$, $C_{a,\alpha}C_{a',\alpha'}=J_{q^2}$.
\item For $\alpha,\alpha',\alpha''\in\mathbb{F}_q$, $(I_q\otimes \phi(\alpha''))C_{\alpha,\alpha'}=C_{\alpha,\alpha'+\alpha''}$. 
\item For $\alpha,\alpha'\in\mathbb{F}_q$, $(I_q\otimes \phi(\alpha))C_{y,\alpha'}=C_{y,\alpha'}$. 
\item $\sum_{a,b\in \mathbb{F}_q\cup\{y\},a\neq b} P_{a}P_{b}=2q(J_{2q+3}-I_{2q+3})$. \end{enumerate}
\end{lemma}
\begin{proof}
(1) is easy to see.   
(2): For $\alpha,\beta,\beta'\in\mathbb{F}_q$, the $(\beta,\beta')$-entry of $\sum_{\gamma\in\mathbb{F}_q}C_{\gamma,\alpha}$ is 
\begin{align*}
\sum_{\gamma\in\mathbb{F}_q}\phi(\gamma(-\beta+\beta')+\alpha)&=\begin{cases}\sum_{\gamma\in\mathbb{F}_q}\phi(\alpha) & \text{ if } \beta=\beta' \\ \sum_{\gamma'\in\mathbb{F}_q}\phi(\gamma'+\alpha) & \text{ if } \beta\neq \beta' \end{cases}\\
&=\begin{cases}q\phi(\alpha) & \text{ if } \beta=\beta', \\ J_q & \text{ if } \beta\neq \beta', \end{cases}
\end{align*}
which yields $\sum_{\gamma\in\mathbb{F}_q}C_{\gamma,\alpha}=q I_q\otimes \phi(\alpha)+(J_q-I_q)\otimes J_q$. 
Therefore,
\begin{align*}
\sum_{a\in\mathbb{F}_q\cup\{y\}}C_{a,\alpha}=\sum_{\gamma\in\mathbb{F}_q}C_{\gamma,\alpha}+C_{y,\alpha}
=q I_q\otimes \phi(\alpha)+(J_q+\phi(\alpha)-I_q)\otimes J_q.
\end{align*} 

(3): For $a=y$, $C_{y,\alpha}C_{y,\alpha'}=(\phi(\alpha)\otimes J_q)(\phi(\alpha')\otimes J_q)=q\phi(\alpha+\alpha')\otimes J_q=qC_{y,\alpha+\alpha'}$. 
For $a,\beta,\beta'\in\mathbb{F}_q$, the $(\beta,\beta')$-entry of $C_{a,\alpha}C_{a,\alpha'}$ is 
\begin{align*}
\sum_{\gamma\in\mathbb{F}_q}\phi(a(-\beta+\gamma)+\alpha)\phi(a(-\gamma+\beta')+\alpha')
&=\sum_{\gamma\in\mathbb{F}_q}\phi(a(-\beta+\beta')+\alpha+\alpha')\\
&=q\phi(a(-\beta+\beta')+\alpha+\alpha'). 
\end{align*}
Thus we have $C_{a,\alpha}C_{a,\alpha'}=q C_{a,\alpha+\alpha'}$. 

(4): The case of $a\in\mathbb{F}_q$ and $a'=y$ follows from the fact that $C_{a,\alpha}$ is a block matrix whose $q\times q$ sub-block is a permutation matrix. The case of $a,a'\in\mathbb{F}_q,a\neq a'$ follows from a similar calculation to (2) with the fact that $ \{(a-a')\gamma : \gamma\in\mathbb{F}_q\}=\mathbb{F}_q$. 

(5) and (6) are routine, and (7) follows from the equations below. Recall that $S=\mathbb{F}_q\cup\{x,y\}$. 
\begin{align*}
\sum_{a,b\in \mathbb{F}_q\cup\{y\},a\neq b} P_{a}P_{b}&=\sum_{a,b\in \mathbb{F}_q\cup\{y\}} P_{a}P_{b}-\sum_{a\in \mathbb{F}_q\cup\{y\}} P_{a}^2\\
&=(\sum_{a\in \mathbb{F}_q\cup\{y\}} P_{a})^2-\sum_{a\in \mathbb{F}_q\cup\{y\}} (2I_{2q+3}+V^{2\varphi(a)}+V^{-2\varphi(a)})\\
&=(J_{2q+3}-I_{2q+3})^2-2(q+1)I_{2q+3}-(J_{2q+3}-I_{2q+3})\\
&= 2q(J_{2q+3}-I_{2q+3}).
\end{align*}
This completes the proof.
\end{proof}

We are now ready  to prove the results for $N_{\alpha}$'s. 
%Note that the result for $N_0$ being a symmetric $((q+2)q^2,q^2+q,q)$-design is well-known, see for example \cite[Exercise~5.7]{S}. 
\begin{thm}\label{thm:1}
\begin{enumerate}
\item  For any $\alpha\in\mathbb{F}_q$, $N_{\alpha}^t=N_{-\alpha}$. 
\item   For any $\alpha,\beta\in\mathbb{F}_q$, 
\begin{align*}
N_{\alpha} N_{\beta}
&= 2q^2 I_{q(2q+3)}\otimes \phi(\alpha+\beta)+2qI_{2q+3}\otimes (\phi(\alpha+\beta)-I_{q})\otimes J_q+2q J_{q^2(2q+3)}\\
&\quad + q\sum_{a\in\mathbb{F}_q\cup\{y\}}(V^{2\varphi(a)}+V^{-2\varphi(a)})\otimes  C_{a,\alpha+\beta}. 
\end{align*}
\end{enumerate}
\end{thm}
\begin{proof}
%\begin{enumerate}
% \item 
(1): It follows from Lemma~\ref{lem:1} (1) and the properties there that the matrices $L$ are symmetric. 

(2):  We use Lemma~\ref{lem:1} to obtain: 
%\end{enumerate}
\begin{align*}
N_{\alpha} N_{\beta} &%=(\sum_{a\in\mathbb{F}_q\cup\{y\}} P_a\otimes C_{a,\alpha})(\sum_{b\in \mathbb{F}_q\cup\{y\}} P_{b}\otimes C_{b,\beta})
=\sum_{a,b\in \mathbb{F}_q\cup\{y\}} P_{a}P_{b}\otimes C_{a,\alpha}C_{b,\beta}\\
&=\sum_{a\in\mathbb{F}_q\cup\{y\}} P_{a}^2\otimes C_{a,\alpha}C_{a,\beta}+\sum_{a,b\in\mathbb{F}_q\cup\{y\},a\neq b} P_{a}P_{b}\otimes C_{a,\alpha}C_{b,\beta}\\
&=\sum_{a\in\mathbb{F}_q\cup\{y\}}(2 I_{2q+3}+V^{2\varphi(a)}+V^{-2\varphi(a)})\otimes q C_{a,\alpha+\beta}+\sum_{a,b\in\mathbb{F}_q\cup\{y\},a\neq b} P_{a}P_{b}\otimes J_{q^2}\displaybreak[0]\\
&=2q I_{2q+3}\otimes (qI_q\otimes \phi(\alpha+\beta)+(J_q+\phi(\alpha+\beta)-I_q)\otimes J_q)\\
&\quad+q\sum_{a\in\mathbb{F}_q\cup\{y\}}(V^{2\varphi(a)}+V^{-2\varphi(a)})\otimes  C_{a,\alpha+\beta}\\
&\quad+2q(J_{2q+3}-I_{2q+3})\otimes J_{q^2}\displaybreak[0]\\
&=2q^2 I_{q(2q+3)}\otimes \phi(\alpha+\beta)+2qI_{2q+3}\otimes (\phi(\alpha+\beta)-I_{q})\otimes J_q+2q J_{q^2(2q+3)}\\
&\quad + q\sum_{a\in\mathbb{F}_q\cup\{y\}}(V^{2\varphi(a)}+V^{-2\varphi(a)})\otimes  C_{a,\alpha+\beta}.
\end{align*}
This completes the proof. 
\end{proof}
\begin{cor}
\begin{enumerate}
\item For $\alpha\in\mathbb{F}_q$,  $N_{\alpha}N_{\alpha}^t=2q^2 I_{q^2(2q+3)}+2q J_{q^2(2q+3)}+
 q\sum_{a\in\mathbb{F}_q\cup\{y\}}(V^{2\varphi(a)}+V^{-2\varphi(a)})\otimes  C_{a,0}$.
\item  $N_0$ is the adjacency matrix of a Deza graph with parameters $(q^2(2q+3),2q^2+2q,3q,2q)$.
\item $N_\alpha$ $(\alpha\in\mathbb{F}_q^*)$ is the adjacency matrix of a directed Deza graph of type II with parameters $(q^2(2q+3),2q^2+2q,3q,2q)$.
\end{enumerate}
\end{cor}
Furthermore, the resulting  directed Deza graphs of type II are commuting with each other. 
Similar construction and a result for prime power $2^m$ were given in \cite{KSS}.  

\bigskip

\vspace*{0.4cm}

\noindent {\bf Acknowledgement} \\
Dean Crnkovi\' c and Andrea \v Svob were supported by {\rm C}roatian Science Foundation under the project 6732. 
Hadi Kharaghani acknowledges the support of the Natural Sciences and Engineering Research Council of Canada (NSERC).
Sho Suda is supported by JSPS KAKENHI Grant Number 18K03395.

\end{document}